\documentclass[a4paper,12pt,reqno]{amsart}
\usepackage{eurosym}
\usepackage{amsfonts}
\usepackage{amsthm}
\usepackage{graphicx, color}
\usepackage{amssymb}
\usepackage{amsmath}

\nonstopmode \numberwithin{equation}{section}

\newtheorem{theorem}{Theorem}

\newtheorem{lemma}{Lemma}[section]

\allowdisplaybreaks

\begin{document}
\title[Properties of generalized multiindex Bessel function]
{Fractional order integration and certain integrals of generalized multiindex Bessel function}

\author{K.S. Nisar$^1$, S.D. Purohit$^{2, *}$, D.L. Suthar$^{3}$, Jagdev Singh$^{4}$}

\address{$^1$Department of Mathematics, College of Arts and
Science-Wadi Al dawaser, Prince Sattam bin Abdulaziz University, Riyadh
region 11991, Saudi Arabia.}
\email{ksnisar1@gmail.com, n.sooppy@psau.edu.sa}

\address{$^2$Department of HEAS (Mathematics), Rajasthan Technical
University, Kota-324010, India.}
\email{sunil\_a\_purohit@yahoo.com}

\address{$^3$Department of Mathematics, Wollo University, Dessie, Ethiopia}
\email{dlsuthar@gmail.com}

\address{$^4$Department of Mathematics, JECRC University, Jaipur, India}
\email{jagdevsinghrathore@gmail.com}

\thanks{$^{*}$ Corresponding author}

\begin{abstract}
We aim to introduce the generalized multiindex Bessel function $J_{\left( \beta _{j}\right) _{m},\kappa ,b}^{\left( \alpha _{j}\right)_{m},\gamma ,c}\left[ z\right] $ and to present some formulas of the Riemann-Liouville fractional integration and differentiation operators. Further, we also derive certain integral formulas involving the newly defined generalized multiindex Bessel function $J_{\left( \beta _{j}\right) _{m},\kappa ,b}^{\left( \alpha _{j}\right)_{m},\gamma ,c}\left[ z\right] $. We prove that such integrals are expressed in terms of the Fox-Wright function $_{p}\Psi_{q}(z)$. The results presented here are of general in nature and easily reducible to new and known results.
\end{abstract}

\maketitle

\section{Introduction and Preliminaries}\label{sec-1}
Fractional calculus, which has a long history, is an important branch of mathematical
analysis (calculus) where differentiations and integrations can be of arbitrary
non-integer order. The operators of Riemann-Liouville fractional integrals and derivatives are defined, for $\alpha \in \mathbb{C}\,\,(\Re(\lambda) > 0)$ and $x>0$
( see, for details, \cite{KST,Samko})%
\begin{equation}
\left( I_{0+}^{\lambda }f\right) \left( x\right) =\frac{1}{\Gamma \left(
\lambda \right) }\int_{0}^{x}\frac{f\left( t\right) }{\left( x-t\right)
^{1-\lambda }}dt\text{ \ },  \label{FI-1}
\end{equation}%
\begin{equation}
\left( I_{-}^{\lambda }f\right) \left( x\right) =\frac{1}{\Gamma \left(
\lambda \right) }\int_{x}^{\infty }\frac{f\left( t\right) }{\left(
t-x\right) ^{1-\lambda }}dt\text{ \ },  \label{FI-2}
\end{equation}%
\begin{eqnarray}
\left( D_{0+}^{\lambda }f\right) \left( x\right) &=&\left( \frac{d}{dx}%
\right) ^{\left[ \Re\left( \lambda \right) \right] +1}\left(
I_{0+}^{1-\lambda +\left[ \Re\left( \lambda \right) \right] }f\right)
\left( x\right)  \notag \\
&=&\left( \frac{d}{dx}\right) ^{\left[ \Re\left( \lambda \right) %
\right] +1}\frac{1}{\Gamma \left( 1-\lambda +\Re\left[ \lambda \right]
\right) }\int_{0}^{x}\frac{f\left( t\right) }{\left( x-t\right) ^{\lambda -%
\left[ \Re\left( \lambda \right) \right] }}\text{ \ \ }
\end{eqnarray}%
and
\begin{eqnarray}
\left( D_{-}^{\lambda }f\right) \left( x\right) &=&\left( -\frac{d}{dx}%
\right) ^{\left[ \Re\left( \lambda \right) \right] +1}\left(
I_{-}^{1-\lambda +\left[ \Re\left( \lambda \right) \right] }f\right)
\left( x\right)  \notag \\
&=&\left( -\frac{d}{dx}\right) ^{\left[ \Re\left( \lambda \right) +1%
\right] }\frac{1}{\Gamma \left( 1-\lambda +\left[ \Re\left( \lambda
\right) \right] \right) }\int_{x}^{\infty }\frac{f\left( t\right) }{\left(
t-x\right) ^{\lambda -\left[ \Re\left( \lambda \right) \right] }}dt~\
  \label{FD-2}
\end{eqnarray}%
respectively,  where $\left[ \Re\left( \lambda \right) \right] $ is the
integral part of $\Re\left( \lambda \right)$. The following lemma is needed in sequel \cite[(2.44)]{Samko},

\begin{lemma}
\label{Lem1} Let $\lambda \in \mathbb{C}$ $\left( \Re\left( \lambda
\right) >0\right) $ and $\delta \in \mathbb{C}$ then

$\left( a\right) ~$If $\ \Re\left( \delta \right) >0$ then
\begin{equation}
\left( I_{0+}^{\lambda }t^{\delta -1}\right) \left( x\right) =\frac{\Gamma
\left( \delta \right) }{\Gamma \left( \lambda +\delta \right) }x^{\lambda
+\delta -1}.  \label{Lem-a}
\end{equation}

$\left(b\right)~$If $\ \Re\left( \delta \right) >\Re\left(
\lambda \right) >0$ then%
\begin{equation}
\left( I_{-}^{\lambda }t^{-\delta }\right) \left( x\right) =\frac{\Gamma
\left( \delta -\lambda \right) }{\Gamma \left( \delta \right) }x^{\lambda
-\delta }.  \label{Lem-b}
\end{equation}
\end{lemma}
In this paper, we aim to introduce a new generalized multiindex Bessel function and to study its compositions with the classical Riemann-Liouville fractional integration and differentiation operators. Further, we derive certain integral formulas involving the newly defined generalized multiindex Bessel function $\mathcal{J}_{\left( \beta _{j}\right) _{m},\kappa ,b}^{\left( \alpha _{j}\right)_{m},\gamma ,c}\left[ z\right] $. We prove that such integrals are expressed in terms of the Fox-Wright function $_{p}\Psi_{q}(z)$.
\section{Fractional Calculus Approach of $\mathcal{J}_{\left( \beta _{j}\right) _{m},\kappa ,b}^{\left( \alpha _{j}\right)_{m},\gamma ,c}\left[ z\right] ~$}
In this section, we introduce a generalized multiindex Bessel function $\mathcal{J}_{\left( \beta _{j}\right) _{m},\kappa ,b}^{\left( \alpha _{j}\right)_{m},\gamma ,c}\left[ z\right] ~$  as follows:

\noindent For $\alpha _{j},\beta _{j},\gamma ,b,c\in \mathbb{C~}\left(j=1,2,...,m\right) $ be such that $\sum\limits_{j=1}^{m}\Re \left( \alpha_{j}\right) >\max 0;$ $ \left\{\Re \left( \kappa \right) -1\right\} ;\kappa>0,\Re \left( \beta _{j}\right) >0$ and $\Re (\gamma )>0$, then

\begin{equation}
\mathcal{J}_{\left( \beta _{j}\right) _{m},\kappa ,b}^{\left( \alpha _{j}\right)
_{m},\gamma ,c}\left[ z\right] =\sum_{n=0}^{\infty }\frac{c^{n}\left( \gamma
\right) _{\kappa n}}{\prod\limits_{j=1}^{m}\Gamma \left( \alpha _{j}n+\beta
_{j}+\frac{b+1}{2}\right) }\, \frac{z^{n}}{n!}~~\left( m\in
\mathbb{N}\right).  \label{GMIBF}
\end{equation}%
Here and in the following, $(\lambda)_\nu$ denotes the Pochhammer symbol  defined (for $\lambda,\,\nu \in \mathbb{C}$), in terms of the Gamma function $\Gamma$ (see \cite[Section 1.1]{Sr-Ch-12}), by
\begin{equation}\label{Poch-symbol}
  (\lambda)_\nu := \frac{\Gamma (\lambda +\nu)}{ \Gamma (\lambda)}
 =\left\{\aligned & 1  \hskip 40 mm (\nu=0;\,\, \lambda \in \mathbb{C}\setminus \{0\}) \\
        & \lambda (\lambda +1) \cdots (\lambda+n-1) \hskip 6mm (\nu=n \in {\mathbb N};\,\, \lambda \in \mathbb{C}).
   \endaligned \right. \\
\end{equation}
\subsection{Fractional Integration}\label{sec-2} 
We first recall the definition of the Fox-Wright function $_{p}\Psi_{q}(z)$ $(p,\, q\in {\mathbb N}_0)$ (see, for details, \cite{Fox3, Wright12}):
  \begin{align}\label{Fox-Wright}
&_{p}\Psi_{q} \left[\begin{array}{c}
 (\alpha_{1},A_{1}),\ldots,(\alpha_{p}, A_{p});          \\
 (\beta_{1},B_{1}),\ldots,(\beta_{q}, B_{q});           \\
\end{array} z \right]= \sum_{n=0}^{\infty} \frac{\Gamma(\alpha_{1}+A_{1}n)\cdots\Gamma(\alpha_{p}+A_{p}n)}{\Gamma(\beta_{1}+B_{1}n)\cdots\Gamma(\beta_{q}+B_{q}n)} \;\frac{z^{n}}{n!}
\end{align}
$$ \left(A_{j} \in \mathbb{R}^+\;(j=1,\ldots,p);\;B_{j}\in \mathbb{R}^+\;(j=1,\ldots,q);\;1+\sum_{j=1}^{q}B_{j}- \sum_{j=1}^{p} A_{j}\geqq0\right), $$
where the equality in the convergence condition holds true for
$$|z|<\nabla:=\left(\prod\limits_{j=1}^{p}A_{j}^{-A_{j}}\right)\;.
\;\left(\prod\limits_{j=1}^{q}B_{j}^{B_{j}}\right).$$

Now we present the Riemann-Liouville fractional integration of the generalized multiindex Bessel function $\mathcal{J}_{\left( \beta _{j}\right) _{m},\kappa ,b}^{\left( \alpha _{j}\right)_{m},\gamma ,c}\left[ z\right] $ in the following theorems.
\begin{theorem}
\label{Th-1} Let $\lambda ,\delta \in \mathbb{C}$ be such that $\Re\left(
\lambda \right) >0,\Re\left( \delta \right) >0$ and the conditions
given in $\left( \ref{GMIBF}\right) $ is satisfied, then for $x>0$, the
following integral formula holds true%
\begin{equation}
\left( I_{0+}^{\lambda }\left\{t^{\delta -1}\mathcal{J}_{\left( \beta _{j}\right)
_{m},\kappa ,b}^{\left( \alpha _{j}\right) _{m},\gamma ,c}\left( t\right)
\right\} \right) \left( x\right) =\frac{x^{\lambda +\delta -1}}{\Gamma \left(
\gamma \right) }~_{2}\Psi _{m+1}\left[
\begin{array}{c}
\left( \gamma ,k\right) ,\left( \delta ,1\right) \\
\left( \beta _{j}+\frac{b+1}{2},\alpha _{j}\right) _{j=1}^{m},\left( \lambda
+\delta ,1\right)%
\end{array}%
\left\vert cx\right. \right] .  \label{Th1-Eqn1}
\end{equation}%
\end{theorem}

\begin{proof}
Let us denote the left-hand side of $\left( \ref{Th1-Eqn1}%
\right) $ by $\mathcal{I}_{1}$. Using the definition $\left( \ref{GMIBF}\right) $,  we have
\begin{eqnarray}
\mathcal{I}_{1} &=&\left( I_{0+}^{\lambda }\left\{ t^{\delta -1}\mathcal{J}_{\left(
\beta _{j}\right) _{m},\kappa ,b}^{\left( \alpha _{j}\right) _{m},\gamma
,c}\left( t\right)\right\} \right) \left( x\right)  \notag \\
&=&\left( I_{0+}^{\lambda }\left\{ t^{\delta -1}\sum_{n=0}^{\infty }\frac{%
c^{n}\left( \gamma \right) _{\kappa n}}{\prod\limits_{j=1}^{m}\Gamma \left(
\alpha _{j}n+\beta _{j}+\frac{b+1}{2}\right) }\frac{t^{n}}{n!}\right\}
\right) \left( x\right).  \label{Th1-Eqn2}
\end{eqnarray}%
Interchanging the integration and the summation in $\left( \ref{Th1-Eqn2}%
\right) $ and using the definition of Pochhammer symbol $\left( \ref{Poch-symbol}\right) $, we get%
\begin{equation*}
\mathcal{I}_{1}=\sum_{n=0}^{\infty }\frac{c^{n}\Gamma \left( \gamma +\kappa
n\right) }{\Gamma \left( \gamma \right) \prod\limits_{j=1}^{m}\Gamma \left(
\alpha _{j}n+\beta _{j}+\frac{b+1}{2}\right) n!}\left( I_{0+}^{\lambda
}t^{\delta +n-1}\right) \left( x\right).
\end{equation*}%
Applying the relation $\left( \ref{Lem-a}\right) $ in Lemma \ref{Lem1}, we get
\begin{equation*}
\mathcal{I}_{1}=\sum_{n=0}^{\infty }\frac{c^{n}\Gamma \left( \gamma +\kappa
n\right) }{\Gamma \left( \gamma \right) \prod\limits_{j=1}^{m}\Gamma \left(
\alpha _{j}n+\beta _{j}+\frac{b+1}{2}\right) n!}\frac{\Gamma \left( \delta
+n\right) }{\Gamma \left( \lambda +\delta +n\right) }x^{\lambda +\delta
+n-1}.
\end{equation*}%
In view of the definition of the Fox-Wright function $\left( \ref{Fox-Wright}\right) $, we arrived at the desired
result.
\end{proof}

\begin{theorem}
\label{Th-2} Let $\lambda ,\delta \in \mathbb{C}$ such that $\Re\left(
\delta \right) >\Re\left( \lambda \right) >0$ and the conditions given
in $\left( \ref{GMIBF}\right) $ is satisfied, then for $x>0$, the following
integral formula holds true%
\begin{equation}
\left( I_{-}^{\lambda }\left\{ t^{-\delta }\mathcal{J}_{\left( \beta _{j}\right)
_{m},\kappa ,b}^{\left( \alpha _{j}\right) _{m},\gamma ,c}\left( \frac{1}{t}%
\right) \right\} \right) \left( x\right) =\frac{x^{\lambda -\delta }}{\Gamma
\left( \gamma \right) }~_{2}\Psi _{m+1}\left[
\begin{array}{c}
\left( \gamma ,k\right) ,\left( \delta -\lambda ,1\right) \\
\left( \beta _{j}+\frac{b+1}{2},\alpha _{j}\right) _{j=1}^{m},\left( \delta
,1\right)%
\end{array}%
\left\vert \frac{c}{x}\right. \right] .  \label{Th2-Eqn1}
\end{equation}
\end{theorem}

\begin{proof}
Denoting the left-hand side of $\left( \ref{Th1-Eqn2}\right) $ by $\mathcal{I%
}_{2}$. Using $\left( \ref{GMIBF}\right) ,$ we have
\begin{eqnarray}
\mathcal{I}_{2} &=&\left( I_{-}^{\lambda }\left\{ t^{-\delta }\mathcal{J}_{\left( \beta
_{j}\right) _{m},\kappa ,b}^{\left( \alpha _{j}\right) _{m},\gamma ,c}\left(
\frac{1}{t}\right) \right\} \right) \left( x\right)  \notag \\
&=&\left( I_{-}^{\lambda }\left\{ t^{-\delta }\sum_{n=0}^{\infty }\frac{%
c^{n}\left( \gamma \right) _{\kappa n}}{\prod\limits_{j=1}^{m}\Gamma \left(
\alpha _{j}n+\beta _{j}+\frac{b+1}{2}\right) }\frac{t^{-n}}{n!}\right\}
\right) \left( x\right) ~.~  \label{Th2-Eqn2}
\end{eqnarray}%
Interchanging the integration and the summation in $\left( \ref{Th2-Eqn2}%
\right) $ and using the definition of Pochhammer symbol $\left( \ref{Poch-symbol}\right) $, we get%
\begin{equation*}
\mathcal{I}_{2}=\sum_{n=0}^{\infty }\frac{c^{n}\Gamma \left( \gamma +\kappa
n\right) }{\Gamma \left( \gamma \right) \prod\limits_{j=1}^{m}\Gamma \left(
\alpha _{j}n+\beta _{j}+\frac{b+1}{2}\right) n!}\left( I_{-}^{\lambda
}t^{-\delta -n}\right) \left( x\right).
\end{equation*}%
Applying the relation $\left( \ref{Lem-b}\right) $ in Lemma \ref{Lem1}, we get
\begin{equation*}
\mathcal{I}_{1}=\sum_{n=0}^{\infty }\frac{c^{n}\Gamma \left( \gamma +\kappa
n\right) }{\Gamma \left( \gamma \right) \prod\limits_{j=1}^{m}\Gamma \left(
\alpha _{j}n+\beta _{j}+\frac{b+1}{2}\right) n!}\frac{\Gamma \left( \delta
+n-\lambda \right) }{\Gamma \left( \delta +n\right) }x^{\lambda -\delta -n}.
\end{equation*}%
In view of the definition of the Fox-Wright function $\left( \ref{Fox-Wright}\right) $, we arrived at the desired
result.
\end{proof}

\subsection{Fractional differentiation}

In this subsection, we establish the fractional differentiation of generalized
multiindex Bessel function given in $\left( \ref{GMIBF}\right).$

\begin{theorem}
\label{Th-3} Let $\lambda ,\delta \in \mathbb{C}$ such that $\Re\left(
\lambda \right) >0,\Re\left( \delta \right) >0$ and the conditions
given in $\left( \ref{GMIBF}\right) $ is satified, then for $x>0$, the following fractional differentiation formula holds true
\begin{equation}
\left( D_{0+}^{\lambda }\left\{ t^{\delta -1}\mathcal{J}_{\left( \beta _{j}\right)
_{m},\kappa ,b}^{\left( \alpha _{j}\right) _{m},\gamma ,c}\left( t\right)
\right\} \right) \left( x\right) =\frac{x^{\delta -\lambda -1}}{\Gamma \left(
\gamma \right) }~_{2}\Psi _{m+1}\left[
\begin{array}{c}
\left( \gamma ,k\right) ,\left( \delta ,1\right) \\
\left( \beta _{j}+\frac{b+1}{2},\alpha _{j}\right) _{i=1}^{m},\left( \delta
-\lambda ,1\right)%
\end{array}%
\left\vert cx\right. \right] .  \label{Th3-Eqn1}
\end{equation}
\end{theorem}

\begin{proof}
Let $\mathcal{I}_{3}$ denote the left-hand side of $\left( \ref{Th3-Eqn1}%
\right) $. Using the definition $\left( \ref{GMIBF}\right)$, we have
\begin{eqnarray*}
\mathcal{I}_{3} &=&\left( D_{0+}^{\lambda }\left\{ t^{\delta -1}\mathcal{J}_{\left(
\beta _{j}\right) _{m},\kappa ,b}^{\left( \alpha _{j}\right) _{m},\gamma
,c}\left( t\right) \right\} \right) \left( x\right) \\
&=&\left( \frac{d}{dx}\right) ^{n}\left( I_{0+}^{n-\lambda }\left\{ t^{\delta
-1}\sum_{r=0}^{\infty }\frac{c^{r}\left( \gamma \right) _{\kappa r}}{%
\prod\limits_{j=1}^{m}\Gamma \left( \alpha _{j}r+\beta _{j}+\frac{b+1}{2}%
\right) }\frac{t^{r}}{r!}\right\} \right) \left( x\right) ~,~ \\
&=&\left( \frac{d}{dx}\right) ^{n}\sum_{r=0}^{\infty }\frac{c^{r}\left(
\gamma \right) _{\kappa r}}{\prod\limits_{j=1}^{m}\Gamma \left( \alpha
_{j}r+\beta _{j}+\frac{b+1}{2}\right) r!}\left( I_{0+}^{n-\lambda }t^{\delta
+r-1}\right) \left( x\right) .
\end{eqnarray*}%
Using the relation $\left( \ref{Lem-a}\right) $ and the definition of the
Pochhammer symbol $\left( \ref{Poch-symbol}\right) $, we get
\begin{eqnarray*}
\mathcal{I}_{3} &=&\left( \frac{d}{dx}\right) ^{n}\sum_{r=0}^{\infty }\frac{%
c^{r}\Gamma \left( \gamma +\kappa r\right) }{\Gamma \left( \gamma \right)
\prod\limits_{j=1}^{m}\Gamma \left( \alpha _{j}r+\beta _{j}+\frac{b+1}{2}%
\right) r!} \frac{\Gamma \left( \delta +r\right) }{\Gamma \left( n-\lambda
+\delta +r\right) }x^{n-\lambda +\delta +r-1}.
\end{eqnarray*}%
By Interchanging the differentiation and the summation, we get
\begin{eqnarray*}
\mathcal{I}_{3} &=&\sum_{r=0}^{\infty }\frac{c^{r}\Gamma \left( \gamma
+\kappa r\right) }{\Gamma \left( \gamma \right) \prod\limits_{j=1}^{m}\Gamma
\left( \alpha _{j}r+\beta _{j}+\frac{b+1}{2}\right) r!} \frac{\Gamma \left( \delta +r\right) }{\Gamma \left( n-\lambda
+\delta +r\right) }\left( \frac{d}{dx}\right) ^{n}x^{n-\lambda +\delta +r-1}
\\
&=&\frac{1}{\Gamma \left( \gamma \right) }\sum_{r=0}^{\infty }\frac{%
c^{r}\Gamma \left( \gamma +\kappa r\right) }{\prod\limits_{j=1}^{m}\Gamma
\left( \alpha _{j}r+\beta _{j}+\frac{b+1}{2}\right) r!}\frac{\Gamma \left( \delta +r\right)\Gamma \left( n-\lambda +\delta +r\right) }{\Gamma \left( n-\lambda
+\delta +r\right)\Gamma
\left( \delta -\lambda +r\right) }x^{\delta -\lambda +\delta +r-1}.
\end{eqnarray*}%
In view of the definition of the Fox-Wright function $\left( \ref{Fox-Wright}\right) $, we arrived at the desired
result.
\end{proof}

\begin{theorem}
\label{Th-4} Let $\lambda ,\delta \in \mathbb{C}$ such that $\Re\left(
\lambda \right) >0,\Re\left( \delta \right) >\left[ \Re\left(
\lambda \right) \right] +1-\Re\left( \lambda \right) $ and the
conditions given in $\left( \ref{GMIBF}\right) $ is satisfied, then the
fractional differentiation $D_{-}^{\lambda }$ of generalized multiindex
Bessel function is given by
\begin{equation}
\left( D_{-}^{\lambda }\left\{ t^{-\delta }\mathcal{J}_{\left( \beta _{j}\right)
_{m},\kappa ,b}^{\left( \alpha _{j}\right) _{m},\gamma ,c}\left( \frac{1}{t}%
\right) \right\} \right) \left( x\right) =\frac{x^{1-\lambda -\delta }}{%
\Gamma \left( \gamma \right) }~_{2}\Psi _{m+1}\left[
\begin{array}{c}
\left( \gamma ,k\right) ,\left( \lambda +\delta ,1\right) \\
\left( \beta _{j}+\frac{b+1}{2},\alpha _{j}\right) _{j=1}^{m},\left( \delta
,1\right)%
\end{array}%
\left\vert \frac{c}{x}\right. \right] .  \label{Th4-Eqn1}
\end{equation}
\end{theorem}

\begin{proof}
Let $\mathcal{I}_{4}$ denote the left-hand side of $\left( \ref{Th4-Eqn1}%
\right) $. Applying the definition $\left( \ref{GMIBF}\right)$,  we have
\begin{eqnarray*}
\mathcal{I}_{4} &=&\left( D_{-}^{\lambda }\left\{ t^{-\delta }\mathcal{J}_{\left( \beta
_{j}\right) _{m},\kappa ,b}^{\left( \alpha _{j}\right) _{m},\gamma ,c}\left(
\frac{1}{t}\right) \right\} \right) \left( x\right)  \\
&=&\left( -\frac{d}{dx}\right) ^{n}\left( I_{-}^{n-\lambda }\left(
t^{-\delta }\sum_{r=0}^{\infty }\frac{c^{r}\left( \gamma \right) _{\kappa r}%
}{\prod\limits_{j=1}^{m}\Gamma \left( \alpha _{j}r+\beta _{j}+\frac{b+1}{2}%
\right) }\frac{t^{-r}}{r!}\right) \right) \left( x\right) ~,~ \\
&=&\left( -\frac{d}{dx}\right) ^{n}\sum_{r=0}^{\infty }\frac{c^{r}\left(
\gamma \right) _{\kappa r}}{\prod\limits_{j=1}^{m}\Gamma \left( \alpha
_{j}r+\beta _{j}+\frac{b+1}{2}\right) r!}\left( I_{-}^{n-\lambda }t^{-\delta
-r}\right) \left( x\right) ,
\end{eqnarray*}%
Using the relation $\left( \ref{Lem-b}\right) $ and the definition of the
Pochhammer symbol $\left( \ref{Poch-symbol}\right) $, we get
\begin{eqnarray*}
\mathcal{I}_{4} &=&\left( -\frac{d}{dx}\right) ^{n}\sum_{r=0}^{\infty }\frac{%
c^{r}\Gamma \left( \gamma +\kappa r\right) }{\Gamma \left( \gamma \right)
\prod\limits_{j=1}^{m}\Gamma \left( \alpha _{j}r+\beta _{j}+\frac{b+1}{2}%
\right) r!} \frac{\Gamma \left( \delta +r-n+\lambda \right) }{\Gamma \left(
\delta +r\right) }x^{n-\lambda -\delta -r},
\end{eqnarray*}%
By Interchanging the derivatives and the summation, we get
\begin{eqnarray*}
\mathcal{I}_{4} &=&\sum_{r=0}^{\infty }\frac{c^{r}\Gamma \left( \gamma
+\kappa r\right) }{\Gamma \left( \gamma \right) \prod\limits_{j=1}^{m}\Gamma
\left( \alpha _{j}r+\beta _{j}+\frac{b+1}{2}\right) r!}\frac{\Gamma \left( \delta +r-n+\lambda \right) }{\Gamma \left(
\delta +r\right) }\left( -\frac{d}{dx}\right) ^{n}x^{n-\lambda -\delta -r}
\\
&=&\frac{1}{\Gamma \left( \gamma \right) }\sum_{r=0}^{\infty }\frac{%
c^{r}\Gamma \left( \gamma +\kappa r\right) }{\prod\limits_{j=1}^{m}\Gamma
\left( \alpha _{j}r+\beta _{j}+\frac{b+1}{2}\right) r!} \frac{\Gamma \left( \delta +r-n+\lambda \right) }{\Gamma \left(
\delta +r\right) }\frac{\left( -1\right) ^{n}\Gamma \left( n-\lambda
-\delta -r+1\right) }{\Gamma \left( -\lambda -\delta -r+1\right) }%
x^{1-\delta -\lambda -r} \\
&=&\frac{x^{1-\delta -\lambda }}{\Gamma \left( \gamma \right) }%
\sum_{r=0}^{\infty }\frac{c^{r}\text{~}\Gamma \left( \gamma +\kappa r\right)
}{\prod\limits_{j=1}^{m}\Gamma \left( \alpha _{j}r+\beta _{j}+\frac{b+1}{2}%
\right) r!}\frac{\Gamma \left( \lambda +\delta +r\right) }{\Gamma \left(
\delta +r\right) }.
\end{eqnarray*}%
In view of the definition of the Fox-Wright function $\left( \ref{Fox-Wright}\right) $, we arrived at the desired
result.
\end{proof}

\section{Certain Integrals of the $\mathcal{J}_{\left( \beta _{j}\right) _{m},\kappa ,b}^{\left( \alpha _{j}\right)_{m},\gamma ,c}\left[ z\right] ~$}

Recently many researchers are developing a large number of
integral formulas involving a variety of special functions \cite{Nisar-IBMR,Choi-BVP,Choi-BKMS,Choi-KMJ,Garg,Menaria-Sohag,Menaria-Acta,Menaria-SMA,
Nisar-CKMS,Nisar-FJMS,Nisar-BS,Rakha}. In this section,
four integral formulas involving generalized multi-index Bessel function $\mathcal{J}_{\left( \beta _{j}\right) _{m},\kappa
,b}^{\left( \alpha _{j}\right) _{m},\gamma ,c}\left[ z\right] $ are established, which are expressed in terms of the Fox-Wright function. For the present investigation, we need the following result of Oberhettinger \cite{Oberhettinger}
\begin{equation}
\int_{0}^{\infty }\,x^{\mu -1}\left( x+a+\sqrt{x^{2}+2ax}\right)
^{-\lambda }dx  =2\lambda a^{-\lambda }\left( \frac{a}{2}\right) ^{\mu }\frac{\Gamma
\left( 2\mu \right) \Gamma \left( \lambda -\mu \right) }{\Gamma \left(
1+\lambda +\mu \right) },  \label{Ober-formula}
\end{equation}
provided $0<\Re (\mu )<\Re (\lambda )$ and the following integral formula due to
Lavoie \cite{Lavoi}
\begin{equation}
\int_{0}^{1}\,x^{\alpha -1}\left( 1-x\right) ^{2\beta -1}\left( 1-\frac{x}{%
3}\right) ^{2\alpha -1}\left( 1-\frac{x}{4}\right) ^{\beta -1}dx  =\left( \frac{2}{3}\right) ^{2\alpha }\frac{\Gamma \left( \alpha \right)
\Gamma \left( \beta \right) }{\Gamma \left( \alpha +\beta \right) },\label{Lavoi-formula}
\end{equation}%
with $\Re\left( \alpha \right) >0,\Re\left( \beta \right) >0.$

\begin{theorem}
\label{Th-5} Let $\alpha _{j},\beta _{j},\gamma,,b,c \in \mathbb{C~}\left(
j=1,2,...,m\right) $ be such that $\sum\limits_{j=1}^{m}\Re \left( \alpha
_{j}\right) >\max \left\{ 0;\Re \left( \kappa \right) -1\right\}$ with $\kappa
>0,\Re \left( \beta \right) >-1$, $\Re (\gamma )>0$, $0<\Re \left( \mu
\right) <\Re \left( \lambda +n\right) $ and $x>0~$, then 
\begin{eqnarray}
&&\int_{0}^{\infty }\,x^{\mu -1}\left( x+a+\sqrt{x^{2}+2ax}\right)
^{-\lambda }\mathcal{J}_{\left( \beta _{j}\right) _{m},\kappa ,b}^{\left( \alpha
_{j}\right) _{m},\gamma ,c}\left( \frac{y}{x+a+\sqrt{x^{2}+2ax}}\right) dx
\notag \\
&=&\frac{2^{1-\mu }a^{-\lambda +\mu }\Gamma \left( 2\mu \right) }{\Gamma
\left( \gamma \right) }\ _{3}\Psi _{m+2}\left[
\begin{array}{c}
\left( \gamma ,k\right) ,\left( \lambda +1,1\right) ,\left( \lambda -\mu
,1\right)  \\
\left( \beta _{j}+\frac{b+1}{2},\alpha _{j}\right) _{j=1}^{m},\left( \lambda
,1\right) ,\left( 1+\lambda +\mu ,1\right)
\end{array}%
\left\vert \frac{-cy}{a}\right. \right] .  \label{Th5-Eqn1}
\end{eqnarray}
\end{theorem}

\begin{proof}
Let us denote the right-hand side of $\left( \ref{Th5-Eqn1}\right) $ by $\mathcal{I}%
_{5}$ and using the definition $\left( \ref{GMIBF}\right) ,$ we have
\begin{eqnarray*}
\mathcal{I}_{1} &=&\int_{0}^{\infty }\,x^{\mu -1}\left( x+a+\sqrt{x^{2}+2ax}%
\right) ^{-\lambda }\mathcal{J}_{\left( \beta _{j}\right) m,\kappa ,b}^{\left( \alpha
_{j}\right) m,\gamma ,c}\left( \frac{y}{x+a+\sqrt{x^{2}+2ax}}\right) dx \\
&=&\int_{0}^{\infty }\,x^{\mu -1}\left( x+a+\sqrt{x^{2}+2ax}\right)
^{-\lambda } \\
&&\hskip 15mm\times \sum_{n=0}^{\infty }\frac{\left( -c\right) ^{n}\left( \gamma
\right) _{\kappa n}}{n!\prod\limits_{j=1}^{m}\Gamma \left( \alpha
_{j}n+\beta _{j}+\frac{b+1}{2}\right) }\left( \frac{y}{x+a+\sqrt{x^{2}+2ax}}%
\right) ^{n}.
\end{eqnarray*}%
Interchanging the integration and summation under the suitable convergence
condition gives%
\begin{equation}
\mathcal{I}_{1}=\sum_{n=0}^{\infty }\frac{\left( -c\right) ^{n}\left( \gamma
\right) _{\kappa n}y^{n}}{n!\prod\limits_{j=1}^{m}\Gamma \left( \alpha
_{j}n+\beta _{j}+\frac{b+1}{2}\right) }\int_{0}^{\infty }\,x^{\mu -1}\left(
x+a+\sqrt{x^{2}+2ax}\right) ^{-\left( \lambda +n\right) }dx,  \label{Pf5-Eq1}
\end{equation}%
Applying $\left( \ref{Ober-formula}\right) $ in $\left( \ref{Pf5-Eq1}\right)
$, we get%
\begin{eqnarray*}
\mathcal{I}_{1} &=&\sum_{n=0}^{\infty }\frac{\left( -c\right) ^{n}\left(
\gamma \right) _{\kappa n}y^{n}}{n!\prod\limits_{j=1}^{m}\Gamma \left(
\alpha _{j}n+\beta _{j}+\frac{b+1}{2}\right) }2\left( \lambda +n\right)
a^{-\left( \lambda +n\right) }\left( \frac{a}{2}\right) ^{\mu }\frac{\Gamma \left( 2\mu \right) \Gamma \left( \lambda +n-\mu
\right) }{\Gamma \left( 1+\lambda +\mu +n\right) },
\end{eqnarray*}%
provided $\Re \left( \lambda +n\right) >\Re \left( \mu \right) >0$. Now using the definition of Pochhammer symbol, we get%
\begin{eqnarray*}
\mathcal{I}_{1} &=&\frac{2^{1-\mu }a^{-\lambda +\mu }\Gamma \left( 2\mu
\right) }{\Gamma \left( \gamma \right) }\sum_{n=0}^{\infty }\frac{\Gamma
\left( \gamma +\kappa n\right) }{\prod\limits_{j=1}^{m}\Gamma \left( \alpha
_{j}n+\beta _{j}+\frac{b+1}{2}\right) }\frac{\Gamma \left( \lambda +n+1\right)\Gamma \left( \lambda -\mu +n\right)  }{\Gamma \left( \lambda
+n\right)\Gamma \left(
1+\lambda +\mu +n\right)  }\frac{\left( -\frac{cy}{a}\right) ^{n}}{n!}.
\end{eqnarray*}%
In view of the definition of Fox-Wright function $\left( \ref{Fox-Wright}\right) $, we arrived the desired
result.
\end{proof}

\begin{theorem}
\label{Th-6} Let $\alpha _{j},\beta _{j},\gamma,b,c \in \mathbb{C~}\left(
j=1,2,...,m\right) $ be such that $\sum\limits_{j=1}^{m}\Re \left( \alpha
_{j}\right) >\max \left\{ 0;\Re \left( \kappa \right) -1\right\} $ with $\kappa
>0,\Re \left( \beta \right) >-1$, $\Re (\gamma )>0$, $0<\Re \left( \mu
+n\right) <\Re \left( \lambda +n\right) $ and x$>0~$, then 
\begin{eqnarray}
&&\int_{0}^{\infty }\,x^{\mu -1}\left( x+a+\sqrt{x^{2}+2ax}\right)
^{-\lambda }\mathcal{J}_{\left( \beta _{j}\right) _{m},\kappa ,b}^{\left( \alpha
_{j}\right) _{m},\gamma ,c}\left( \frac{xy}{x+a+\sqrt{x^{2}+2ax}}\right) dx
\notag \\
&=&\frac{2^{1-\mu }a^{-\lambda +\mu }\Gamma \left( 2\mu \right) }{\Gamma
\left( \gamma \right) }\ _{3}\Psi _{m+2}\left[
\begin{array}{c}
\left( \gamma ,k\right) ,\left( \lambda +1,1\right) ,\left( 2\mu ,2\right)
\\
\left( \beta _{j}+\frac{b+1}{2},\alpha _{j}\right) _{j=1}^{m},\left( \lambda
,1\right) ,\left( 1+\lambda +\mu ,2\right)
\end{array}%
\left\vert \frac{-cy}{a}\right. \right] .  \label{Th6-Eq1}
\end{eqnarray}
\end{theorem}

\begin{proof}
Let us denote the right-hand side of $\left( \ref{Th6-Eq1}\right) $ by $\mathcal{I}_{6}
$ and using the definition $\left( \ref{GMIBF}\right) ,$ we have
\begin{eqnarray*}
\mathcal{I}_{6} &=&\int_{0}^{\infty }\,x^{\mu -1}\left( x+a+\sqrt{x^{2}+2ax}%
\right) ^{-\lambda }\mathcal{J}_{\left( \beta _{j}\right) m,\kappa ,b}^{\left( \alpha
_{j}\right) m,\gamma ,c}\left( \frac{xy}{x+a+\sqrt{x^{2}+2ax}}\right) dx \\
&=&\int_{0}^{\infty }\,x^{\mu -1}\left( x+a+\sqrt{x^{2}+2ax}\right)
^{-\lambda } \\
&&\times \sum_{n=0}^{\infty }\frac{\left( -c\right) ^{n}\left( \gamma
\right) _{\kappa n}}{n!\prod\limits_{j=1}^{m}\Gamma \left( \alpha
_{j}n+\beta _{j}+\frac{b+1}{2}\right) }\left( \frac{xy}{x+a+\sqrt{x^{2}+2ax}}%
\right) ^{n}.
\end{eqnarray*}%
Interchanging the integration and summation under the given condition,  yields%
\begin{equation}
\mathcal{I}_{6}=\sum_{n=0}^{\infty }\frac{\left( -c\right) ^{n}\left( \gamma
\right) _{\kappa n}y^{n}}{n!\prod\limits_{j=1}^{m}\Gamma \left( \alpha
_{j}n+\beta _{j}+\frac{b+1}{2}\right) }\int_{0}^{\infty }\,x^{\mu
+n-1}\left( x+a+\sqrt{x^{2}+2ax}\right) ^{-\left( \lambda +n\right) }dx.
\label{Pf2-Eq2}
\end{equation}%
Applying $\left( \ref{Ober-formula}\right) $ on $\left( \ref{Pf2-Eq2}\right)
$, we get%
\begin{eqnarray*}
\mathcal{I}_{6} &=&\sum_{n=0}^{\infty }\frac{\left( -c\right) ^{n}\left(
\gamma \right) _{\kappa n}y^{n}}{n!\prod\limits_{j=1}^{m}\Gamma \left(
\alpha _{j}n+\beta _{j}+\frac{b+1}{2}\right) }2\left( \lambda +n\right)
a^{-\left( \lambda +n\right) }\left( \frac{a}{2}\right) ^{\mu +n}\frac{\Gamma \left( 2\mu +2n\right) \Gamma \left( \lambda -\mu
\right) }{\Gamma \left( 1+\lambda +\mu +2n\right) },
\end{eqnarray*}%
provided $\Re \left( \lambda +n\right) >\Re \left( \mu +n\right) >0$

In view of definition of Pochhammer symbol $\left( \ref{Poch-symbol}\right) $, we get%
\begin{eqnarray*}
\mathcal{I}_{6} &=&\frac{2^{1-\mu }a^{-\lambda +\mu }\Gamma \left( \lambda
-\mu \right) }{\Gamma \left( \gamma \right) }\sum_{n=0}^{\infty }\frac{%
\Gamma \left( \gamma +\kappa n\right) }{\prod\limits_{j=1}^{m}\Gamma \left(
\alpha _{j}n+\beta _{j}+\frac{b+1}{2}\right) }\frac{\Gamma \left( \lambda +n+1\right)\Gamma \left( 2\mu +2n\right)  }{\Gamma \left( \lambda
+n\right) \Gamma \left( 1+\lambda
+\mu +2n\right)}\frac{\left( -\frac{cy}{2}\right) ^{n}}{n!}.
\end{eqnarray*}%
Using the definition of Fox-Wright function $\left( \ref{Fox-Wright}\right) $, we arrived the desired result.
\end{proof}

\begin{theorem}
\label{Th-7} For $\xi ,\sigma\in \mathbb{C}$ with $\Re\left( \xi +\sigma\right)
>0,\Re\left( \xi +n\right) >0$ and then for $x>0,$
\begin{eqnarray*}
&&\int_{0}^{1}\,x^{\xi +\sigma-1}\left( 1-x\right) ^{2\xi -1}\left( 1-\frac{x}{3}%
\right) ^{2\left( \xi +\sigma\right) -1}\left( 1-\frac{x}{4}\right) ^{\xi
-1}\mathcal{J}_{\left( \beta _{j}\right) _{m},\kappa ,b}^{\left( \alpha _{j}\right)
_{m},\gamma ,c}\left( y\left( 1-\frac{x}{4}\right) \left( 1-x\right)
^{2}\right) dx \\
&=&\frac{\Gamma \left( \xi +\sigma\right) }{\Gamma \left( \gamma \right) }\left(
\frac{2}{3}\right) ^{2\left( \xi +\sigma\right) }~_{2}\Psi _{m+1}\left[
\begin{array}{c}
\left( \gamma ,k\right) ,\left( \xi ,1\right)  \\
\left( \beta _{j}+\frac{b+1}{2},\alpha _{j}\right) _{j=1}^{m},\left( 2\xi
+\sigma,1\right)
\end{array}%
\left\vert cy\right. \right] .
\end{eqnarray*}
\end{theorem}

\begin{proof}
Denoting the left-hand side of theorem by $\mathcal{I}_{7}$ and using $%
\left( \ref{GMIBF}\right) ,$we get%
\begin{eqnarray*}
\mathcal{I}_{7} &=&\int_{0}^{1}\,x^{\xi +\sigma-1}\left( 1-x\right) ^{2\xi
-1}\left( 1-\frac{x}{3}\right) ^{2\left( \xi +\sigma\right) -1}\left( 1-\frac{x}{4%
}\right) ^{\xi -1} \\
&&\times \mathcal{J}_{\left( \beta _{j}\right) m,\kappa ,b}^{\left( \alpha _{j}\right)
m,\gamma ,c}\left( y\left( 1-\frac{x}{4}\right) \left( 1-x\right)
^{2}\right) dx, \\
&=&\int_{0}^{1}\,x^{\xi +\sigma-1}\left( 1-x\right) ^{2\xi -1}\left( 1-\frac{x}{3}%
\right) ^{2\left( \xi +\sigma\right) -1}\left( 1-\frac{x}{4}\right) ^{\xi -1} \\
&&\times \sum_{n=0}^{\infty }\frac{c^{n}\left( \gamma \right) _{\kappa n}}{%
\prod\limits_{j=1}^{m}\Gamma \left( \alpha _{j}n+\beta _{j}+\frac{b+1}{2}%
\right) }\frac{y^{n}\left( 1-\frac{x}{4}\right) ^{n}\left( 1-x\right) ^{2n}}{%
n!}dx,
\end{eqnarray*}%
Interchanging the integration and summation gives,%
\begin{eqnarray*}
\mathcal{I}_{7} &=&\sum_{n=0}^{\infty }\frac{c^{n}\left( \gamma \right)
_{\kappa n}y^{n}}{n!\prod\limits_{j=1}^{m}\Gamma \left( \alpha _{j}n+\beta
_{j}+\frac{b+1}{2}\right) } \\
&&\times \int_{0}^{1}\,x^{\xi +\sigma-1}\left( 1-x\right) ^{2\left( \xi +n\right)
-1}\left( 1-\frac{x}{3}\right) ^{2\left( \xi +\sigma\right) -1}\left( 1-\frac{x}{4%
}\right) ^{\xi +n-1}dx.
\end{eqnarray*}%
Now using $\left( \ref{Lavoi-formula}\right) $ and the definition of
Pochhammer symbol,%
\begin{eqnarray*}
\mathcal{I}_{7}&=&\sum_{n=0}^{\infty }\frac{c^{n}\Gamma \left( \kappa +\gamma n\right) y^{n}%
}{n!\Gamma \left( \gamma \right) \prod\limits_{j=1}^{m}\Gamma \left( \alpha
_{j}n+\beta _{j}+\frac{b+1}{2}\right) }\left( \frac{2}{3}\right) ^{2\left( \xi +\sigma\right) }\frac{\Gamma
\left( \xi +\sigma\right) \Gamma \left( \xi +n\right) }{\Gamma \left( 2\xi
+\sigma+n\right) }.
\end{eqnarray*}%
Using the definition of Fox-Wright function $\left( \ref{Fox-Wright}\right) $, we obtained the required result.
\end{proof}

\begin{theorem}
\label{Th-8} For $\xi ,\sigma\in \mathbb{C}$ with $\Re\left( \xi +\sigma\right)
>0,\Re\left( \xi +n\right) >0$ then for $x>0$
\begin{eqnarray*}
&&\int_{0}^{1}\,x^{\xi -1}\left( 1-x\right) ^{2\left( \xi +\sigma\right)
-1}\left( 1-\frac{x}{3}\right) ^{2\xi -1}\left( 1-\frac{x}{4}\right)
^{\left( \xi +\sigma\right) -1}\mathcal{J}_{\left( \beta _{j}\right) _{m},\kappa
,b}^{\left( \alpha _{j}\right) _{m},\gamma ,c}\left( yx\left( 1-\frac{x}{3}%
\right) ^{2}\right) dx \\
&=&\frac{\Gamma \left( \xi +\sigma\right) }{\Gamma \left( \gamma \right) }\left(
\frac{2}{3}\right) ^{2\xi }~_{2}\Psi _{m+1}\left[
\begin{array}{c}
\left( \gamma ,k\right) ,\left( \xi ,1\right)  \\
\left( \beta _{j}+\frac{b+1}{2},\alpha _{j}\right) _{j=1}^{m},\left( 2\xi
+\sigma,1\right)
\end{array}%
\left\vert \frac{4cy}{9}\right. \right] .
\end{eqnarray*}
\end{theorem}

\begin{proof}
Taking left-hand side of theorem  by $\mathcal{I}_{8}$ and using $\left( \ref%
{GMIBF}\right) $, we get%
\begin{eqnarray*}
\mathcal{I}_{8} &=&\int_{0}^{1}\,x^{\xi -1}\left( 1-x\right) ^{2\left( \xi
+\sigma\right) -1}\left( 1-\frac{x}{3}\right) ^{2\xi -1}\left( 1-\frac{x}{4}%
\right) ^{\left( \xi +\sigma\right) -1} \\
&&\times \mathcal{J}_{\left( \beta _{j}\right) _{m},\kappa ,b}^{\left( \alpha
_{j}\right) _{m},\gamma ,c}\left( yx\left( 1-\frac{x}{3}\right) ^{2}\right)
dx, \\
&=&\int_{0}^{1}\,x^{\xi -1}\left( 1-x\right) ^{2\left( \xi +\sigma\right)
-1}\left( 1-\frac{x}{3}\right) ^{2\xi -1}\left( 1-\frac{x}{4}\right)
^{\left( \xi +\sigma\right) -1} \\
&&\times \sum_{n=0}^{\infty }\frac{c^{n}\left( \gamma \right) _{\kappa n}}{%
\prod\limits_{j=1}^{m}\Gamma \left( \alpha _{j}n+\beta _{j}+\frac{b+1}{2}%
\right) }\frac{x^{n}y^{n}\left( 1-\frac{x}{3}\right) ^{2n}}{n!}dx,
\end{eqnarray*}%
Interchanging the integration and summation gives,%
\begin{eqnarray*}
\mathcal{I}_{8} &=&\sum_{n=0}^{\infty }\frac{c^{n}\left( \gamma \right)
_{\kappa n}y^{n}}{n!\prod\limits_{j=1}^{m}\Gamma \left( \alpha _{j}n+\beta
_{j}+\frac{b+1}{2}\right) } \\
&&\times \int_{0}^{1}\,x^{\xi +n-1}\left( 1-x\right) ^{2\left( \xi +\sigma\right)
-1}\left( 1-\frac{x}{3}\right) ^{2\left( \xi +n\right) -1}\left( 1-\frac{x}{4%
}\right) ^{\xi +\sigma-1}dx.
\end{eqnarray*}%
Now using $\left( \ref{Lavoi-formula}\right) $ and the definition of
Pochhammer symbol $\left( \ref{Poch-symbol}\right) $,%
\begin{eqnarray*}
\mathcal{I}_{8}&=&\sum_{n=0}^{\infty }\frac{c^{n}\Gamma \left( \kappa +\gamma n\right) y^{n}%
}{n!\Gamma \left( \gamma \right) \prod\limits_{j=1}^{m}\Gamma \left( \alpha
_{j}n+\beta _{j}+\frac{b+1}{2}\right) } \left( \frac{2}{3}\right) ^{2\xi }\frac{\Gamma \left( \xi +n\right)
\Gamma \left( \xi +\sigma\right) }{\Gamma \left( 2\xi +\sigma+n\right) }.
\end{eqnarray*}%
Using the definition of Fox-Wright function $\left( \ref{Fox-Wright}\right) $, we obtained the desired result.
\end{proof}

\section{Concluding remark and discussion}

The fractional calculus and the integral formulae of the newly defined
generalized multiindex Bessel function are investigated here. Various special cases of the derived results in the paper can be evaluate by taking suitable values of parameters involved. For example, if we set $c=-1$ and  $b=1$ in $\left( \ref{GMIBF}\right) $, we immediately obtain the result due to Choi and Agarwal \cite{Choi-Filomat}:

\begin{equation}
J_{\left( \beta _{j}\right) _{m},\kappa ,1}^{\left( \alpha _{j}\right)
_{m},\gamma ,-1}\left[ z\right] =\sum_{n=0}^{\infty }\frac{\left( \gamma
\right) _{\kappa n}}{\prod\limits_{j=1}^{m}\Gamma \left( \alpha _{j}n+\beta
_{j}+1\right) }\frac{\left( -z\right) ^{n}}{n!}~~\left( m\in \mathbb{N}%
\right) .  \label{MIBF}
\end{equation}
For various other special cases we refer \cite{Choi-Filomat, DSR, DT} and we left results for the interested readers.

{\bf{Conflict of Interests}} The authors declare that there is no conflict of interests
regarding the publication of this paper.\\

\end{document}